%% file: NearlyBoundedWD_v4.tex
\documentclass[a4paper, 12pt, reqno]{amsart}
\usepackage{amsmath,amsfonts,amssymb,graphics,mathrsfs,amsthm,eurosym,verbatim,enumerate,subcaption}
\usepackage[marginratio=1:1,tmargin=117pt, height=650pt]{geometry}
\usepackage[usenames, dvipsnames]{xcolor}

\usepackage[normalem]{ulem}
\usepackage{bm}
\usepackage{tikz-cd}
\usepackage[colorlinks=true,linkcolor=blue!60!black,citecolor=teal!80!black,urlcolor=RoyalBlue]{hyperref}

\def\tef{transcendental entire function}
\numberwithin{equation}{section}
\makeatletter
\def\blfootnote{\xdef\@thefnmark{}\@footnotetext}
\makeatother
\theoremstyle{plain}
\newtheorem{thm}{Theorem}[section]
\newtheorem{proposition}[thm]{Proposition}
\newtheorem{cor}[thm]{Corollary}
\newtheorem{lem}[thm]{Lemma}

\theoremstyle{remark} 
\newtheorem*{remark}{Remark}%[section]
\newcommand{\C}{{\mathbb{C}}}

\newcommand{\N}{{\mathbb{N}}}

\renewcommand{\Re}{\operatorname{Re}}

\newcommand{\dist}{\operatorname{dist}}
\newcommand{\interior}{\operatorname{int}}
\newcommand{\eps}{\varepsilon}

\newcommand*{\defeq}{\mathrel{\vcenter{\baselineskip0.5ex \lineskiplimit0pt
			\hbox{\scriptsize.}\hbox{\scriptsize.}}}%
	=}

\newcommand{\maindisc}{D}
\newcommand{\mainreef}{B_0}					
\newtheorem{Claim}{Claim}
\newenvironment{subproof}{\begin{proof}[Proof of claim.]}{%
\end{proof}}

\begin{document}
	\title[Wandering domains with nearly bounded orbits]{Wandering domains with nearly bounded orbits}
	\author{Leticia Pardo-Sim\'on, \, David J. Sixsmith}
	
	\address{Department of Mathematics \\ The University of Manchester\\
		Manchester \\ M13 9PL\\ UK\textsc{\newline \indent \href{https://orcid.org/0000-0003-4039-5556}{\includegraphics[width=1em,height=1em]{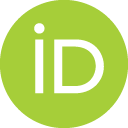} {\normalfont https://orcid.org/0000-0003-4039-5556}}}}
	\email{leticia.pardosimon@manchester.ac.uk}

	\address{School of Mathematics and Statistics\\ The Open University\\
		Milton Keynes MK7 6AA \\ UK \textsc{\newline \indent \href{https://orcid.org/0000-0002-3543-6969}{\includegraphics[width=1em,height=1em]{orcid2.png} {\normalfont https://orcid.org/0000-0002-3543-6969}}}}
	\email{david.sixsmith@open.ac.uk}
	
	\thanks{\vspace{3pt}\vspace{3pt}\\ }
	
	\begin{abstract}
		A major open question in transcendental dynamics asks if it is possible for points in a wandering domain to have bounded orbits, and more strongly, for a wandering domain to iterate only in a bounded domain. In this paper we give a partial answer to this question, by constructing a bounded wandering domain that spends, in a precise sense, nearly all of its time iterating in a bounded domain. This is in strong contrast to all previously known examples of wandering domains.
	\end{abstract}
	\maketitle
	\section{Introduction}
	Let $f \colon \C \to \C$ be an entire function, and let $f^n$ denote the $n$th iterate of the function $f$, for $n\geq 0$. We define the \emph{Fatou set} $F(f)$ as the set of points $z \in \C$ where the iterates $\{f^n\}_{n \in \mathbb{N}}$ form a normal family in some neighbourhood of $z$, and the \emph{Julia set} as its complement $J(f) \defeq \C \setminus F(f)$. Roughly speaking, the iterates of $f$ are stable at points in the Fatou set. For an introduction to the properties of these sets, and in particular the dynamics of {\tef}s, see, for example, the well-known survey \cite{Bergweilermerophormic}. 
	
	In this paper we are interested in \emph{wandering domains}. A wandering domain is a component of the Fatou set, $U$, with the property that $f^n(U) \cap f^m(U) \ne \emptyset$ only when $n = m$. A famous result of Sullivan \cite{Sullivan_noWD_85} implies that no polynomial has a wandering domain, making the study of these objects in transcendental dynamics of particular interest. Following Baker's first example of a transcendental entire map with a wandering domain from 1976, \cite{Baker_wandering_76}, numerous further examples have been provided; e.g.~\cite{Baker_wandering84, Herman_84,Erem_Lyubich_pathological_87, fagella_henriksen_Teich_09, bishop2015, Lazebnik_several_17, david_Shishikura_Wandering, Evdoridou_adi_leti_23}. 

	A significant open question in transcendental dynamics asks if it is possible for a point, and thus all points, of a wandering domain to have a bounded orbit. A stronger version of this question is whether there is a wandering Fatou component with bounded orbit. In other words, is there a transcendental entire function $f$ with a wandering domain $U$ such that its forward orbit, $\bigcup_{n \geq 0} f^n(U)$, is bounded? We give a partial answer to this question by constructing an example of a such a wandering domain $U$ which has a \emph{nearly bounded orbit}; there is a bounded domain $D$ such that
	\begin{equation}
		\label{eq.setofns}
		\lim_{k\rightarrow\infty} \frac{\#\{ n \leq k : f^n(U) \subset D \}}{k} = 1.
	\end{equation}
	In other words, the set of natural numbers $n$ for which $f^n(U)$ is contained in $D$ has upper (and lower) \textit{natural density} one. This is in particularly strong contrast to all existing examples of wandering domains, for which the quantity in \eqref{eq.setofns} is equal to zero for any choice of bounded domain $D$.
	
	We construct our example using classical approximation theory, a method first used to construct wandering domains by Eremenko and Lyubich in 1987, \cite{Erem_Lyubich_pathological_87}, and refined more recently in \cite{BEFRS_internal22, BocThaler_21,  lasse_DJ_Erem_22}. As with previous examples, our construction is quite delicate, particularly in the handling of approximation errors. We construct a sequence of entire functions $f_1, f_2, \ldots$ that converges locally uniformly to a transcendental entire function $f$. We control the errors at the $k$-th step by pulling back certain sets under $f_k$, and then use these sets to define $f_{k+1}$, ``holding up'' the images of the wandering domain, near the origin, by using the repelling fixed point of the M\"obius map
%	\begin{equation}
	%	\label{eq.Phidef}
		$\Phi(z) \defeq 3z.$
%	\end{equation}
	
	Our main result is as follows, and is, in fact, somewhat more general.  Recall that a compact set $K\subset \C$ is \textit{full} if $\C\setminus K$ is connected, and that a domain is \textit{regular} if it equals the interior of its closure.
	\begin{thm}
		\label{maintheorem}
		Let  $U$ be a regular domain whose closure in $\C$ is a full compact set, and suppose that $(n_j)_{j \in \mathbb{N}}, (m_j)_{j \in \mathbb{N}}$ are sequences of natural numbers, $(m_j)_{j \in \mathbb{N}}$ being strictly increasing, and set $n_0=m_0=0$. Then there is a bounded domain $D$, and a transcendental entire function $f$ for which $U$ is a wandering domain with the following property. For each $n \in \mathbb{N}$, the set $f^{n}(U)$ is either contained in $D$ or in $\C\setminus \overline{D}$, and $f^{n}(U)\subset D$ if and only if
		%either
		%\begin{equation}
		%	\label{aneq2}
		%	0 \leq n < n_1,
		%\end{equation}
		%or
		\begin{equation}
			\label{aneq}
			\sum_{i=0}^p n_i + m_i < n \leq  \left(\sum_{i=0}^p n_i + m_i\right) + n_{p+1}
		\end{equation}
		for some $p\geq 0$.
	\end{thm}
	Note that \eqref{aneq} implies that $U$ spends $n_1$ iterates inside $D$, then $m_1$ iterates in its complement, then $n_2$ iterates inside $D$, then $m_2$ in its complement, and so on. In particular we have the following easy corollary. 
	\begin{cor}
		\label{c1}
		Suppose that $\lambda \in [0, 1]$. Then there is a transcendental entire function $f$ with a wandering domain $U$ 
		%such that for bounded open set containing the origin, $D$, we have that, 
		and a bounded domain $D$ such that 
		\begin{equation}
			\label{anothereq}
			\lim_{k\rightarrow\infty} \frac{\#\{ n \leq k : f^n(U) \subset D \}}{k} = \lambda.
		\end{equation}
	\end{cor}
	Observe that Corollary \ref{c1} follows from Theorem~\ref{maintheorem} choosing, for example, the sequences,  for $j \in \mathbb{N}$,
	\begin{equation*}
		m_j = j, \quad \text{ and }\quad  n_j\defeq
		\begin{cases}
			j^2,& \lambda=1,\\
			\left\lceil\frac{\lambda}{1-\lambda}\cdot j \right\rceil, &  \text{otherwise.}\\
		\end{cases}
	\end{equation*}
	In the case that $\lambda = 1$, roughly speaking, the wandering domain in Corollary~\ref{c1} spends ``nearly all'' of its iterates in $D$. Clearly, by choosing the sequence $(n_k)$ to tend to infinity quickly, we can ensure the limit in \eqref{anothereq} tends to one as fast as we wish.
	
	It is natural to ask if more ``pathological'' behaviours are possible. Stating the most general possible results is difficult, and not particularly illuminating. %Following arguments from 
	%\cite{BocThaler_21, lasse_DJ_Erem_22}, we can prescribe different \textit{shapes} for our wandering domain $U$. Moreover, our construction can be modified to ensure that the orbit of $U$ accumulates at a given finite collection of points. 
	We restrict ourselves to sketching a proof of the following.
	\begin{thm}
		\label{fancytheorem}
		Let  $U$ be a regular domain whose closure in $\C$ is a full compact set, let $(z_j)_{1 \leq j \leq p}$ be a collection of  distinct points in $\C\setminus \overline{U}$ and let $(\lambda_j)_{1 \leq j \leq p}$ be a finite sequence of positive real numbers whose sum is at most~$1$. Then there is a transcendental entire function $f$ for which $U$ is a wandering domain and so that for any collection of pairwise disjoint domains $(D_j)_{1 \leq j \leq p}$ with $z_j\in D_j$, 
		\begin{equation*}
			\label{anothereqv2}
			\lim_{k\rightarrow\infty} \frac{\#\{ n \leq k : f^n(U) \subset D_j\}}{k} = \lambda_j, \quad\text{ for } 1 \leq j \leq p.
		\end{equation*}
	\end{thm}

\begin{remark} For each wandering domain $U$ constructed in this paper, and for every $R>0$, there exists $n\defeq n(R, U)\in \N$ such that $\inf\{\vert z \vert \colon z\in f^n(U)\}>R$. Thus, it remains an open question whether wandering domains with bounded orbit exist.
\end{remark}

\subsection*{Acknowledgments} The question of whether wandering domains such as in Corollary~\ref{c1} exist was raised when the first author visited the Analysis group at the University of St Andrews. She thanks them for their hospitality. 

	\section{Preliminaries}
	\subsection{Notation}
	We denote the closure of a set $A \subset \C$ by $\overline{A}$, and its interior by  $\operatorname{int}({A})$. For $c \in \C$, we denote the translate of $A$ by
	$
	A + c \defeq \{ z + c : z \in A \}$, and the Euclidean distance between $c$ and $A$ by $\dist(c,A)$. For sets $A, B\subset \C$, we write $A\Subset B$ to indicate that $A$ is \textit{compactly contained} in $B$, that is, $\overline{A}$ is compact and $ \overline{A}\subseteq \operatorname{int}({B})$. For $a\in \C$ and $r>0$, we denote by $D(a,r)$ the open disk of radius $r$ centred at $a$.

	\subsection{Approximation}
	We will use the following stronger version of Runge's classical approximation theorem (\cite{Runge_85}) as stated in \cite[Theorem 4]{BocThaler_21}; see the appendix in \cite{BocThaler_21} for a proof.
	\begin{thm}\label{thm:Runge}
		Let $A_1, \ldots, A_n \subseteq \C$ be pairwise disjoint and full compact sets. For each $1\leq k\leq n$, let $L_k\subset A_k$ be a finite set of points, and $h_k \colon A_k\to \C$ be a holomorphic function. Then for every $\eps >0$, there exists an entire function $f$ such that, for all $1\leq k\leq n$,
			 \begin{align*}
			&|f(z)-h_k(z)|<\eps, \quad\quad\quad\quad\quad \; \text{ for } z\in A_k; \quad  \text{ and } \\
			 & f(z)=h_k(z), \; f'(z)=h'_k(z), \quad \text{ for } z\in L_k.
		\end{align*}
	\end{thm}
	
	Some of our arguments will follow the constructions of wandering domains in \cite{BocThaler_21, lasse_DJ_Erem_22}. We shall borrow from \cite{lasse_DJ_Erem_22} the following two lemmas on approximation of univalent functions and their iterates. 	
	\begin{lem}[{\cite[version of Lemma 2.3]{lasse_DJ_Erem_22}}]
		\label{lem_univalent} 
		Let $U,V\subseteq \C$ be open,  and let $\phi \colon U\to V$ be a conformal isomorphism. Let $A\subseteq U$ be a closed set such that $\dist(A,\partial U)>0$ and $\mu\defeq  \inf_{z\in A}\vert\phi'(z)\vert>0$. Then there is $\eps >0$ with the  following  property: if $f\colon U\to \C$ is holomorphic with $|f(z) - \phi(z)|\leq \eps$ for all $z\in U$, then $f$ is univalent on $A$, with $f(A)\subseteq V$. Moreover, $\vert f'(z)\vert >\mu/2$, for $z\in A$.
	\end{lem}

		\begin{lem}[{\cite[version of Corollary 2.7]{lasse_DJ_Erem_22}}]
		\label{lem_iterates} 
		Let $U\subset\C$ be open, and let $g \colon U\to \C$ be holomorphic. Suppose that $G\subset U$ is open, $K\subseteq G$ is compact and $g^n$ is defined and univalent on $G$ for some $n\geq 1$. Then, for every $\eps>0$, there is $\delta>0$ with the following property. For any holomorphic $f\colon U\to \C$ with $\vert f(z)-g(z)\vert<\delta$ for all $z\in U$, $f^n$ is defined and univalent on $K$, and $\vert f^k(z)-g^k(z)\vert<\eps$ on $K$, for $k\leq n$.
\end{lem}
	
%\begin{obs}\label{obs_convergence}\normalfont Let  $B$ be an open set and let $g\colon B\to \C$ be a continuous function. Suppose that $A$ is a compact set and $g^m(A)\subset B$ for some $m\geq 2$. Let $\{f_j\}$ be a sequence of functions obtained applying Runge's theorem to $f_{j+1}$ and $f_{j}$ in $B$, with $f_0\equiv g$ and $\eps_j\defeq \frac{\eps_0}{2^j}$ and $\eps_0\defeq \dist(f_0^m(A), \partial B)$.  Then $\dist(f^m(A), \partial B)>\eps_0-\eps_0\sum^{\infty}_{n=2}2^n=\eps_0/2.$ In particular, $f^m(A)\subset B$.
%\end{obs}

Finally, we will use the following result on plane topology.	
\begin{lem}[{\cite[Lemma 2.9]{lasse_DJ_Erem_22}}] \label{lem_compacts} Let $K\subseteq \C$ be a compact and full set. Then there exists an infinite sequence $(K_j)_{j\geq 0}$ of compact and full sets such that $K_j\subseteq \operatorname{int}(K_{j-1})$ for all $j\in \N$ and $K=\bigcap_{j\geq 0} K_j$. In addition, each $K_j$ may be chosen to be bounded by a finite disjoint union of closed Jordan curves.	
\end{lem}

\section{Proof of Theorem~\ref{maintheorem}}\label{sec_3}

\subsection{Behaviour near the origin}
We start by describing the behaviour of our map near the origin. Recall the map $\Phi: z \mapsto 3z$.
\begin{lem}
	\label{pullbacklemma}
	There exist $0 < r_1 < 1/9 < r_2 < r_3 < 1/3$ and $0<\eps <r_2/2$ with the following properties. Define
	 \begin{align*}
		\maindisc&\defeq  D(0,r_1), \quad  \text{ and } \\
		\mainreef&\defeq \left\{ z : |z| \in (r_2, r_3) \text{ and } |\operatorname{ arg } z| <  \frac{\pi}{4} \right\}.
	\end{align*}
	If $f$ is holomorphic on $D(0,\frac19)$, with $|f(z) - \Phi(z)| < \eps$, for $z \in \maindisc$, then $f$ is univalent on $\maindisc$. Moreover, for each $n \in \mathbb{N}$ there is a unique component of $f^{-n}(\mainreef)$ contained in $\maindisc$, and these preimage components are pairwise disjoint. 
	
	If, in addition, $f(0)=0$, then there exists a decreasing sequence $\rho_n\to 0$ such that $f^{-n}(B_0)\cap D\subset D(0,\rho_n)$.
\end{lem}
\begin{proof}
Set $r_2 = \frac{5}{27}$, $r_3 = \frac{7}{27}$. Choose $r_1 \in (0,\frac19)$ sufficiently close to $\frac19$ to ensure that $(\maindisc \cup \mainreef) \Subset \Phi(\maindisc)$. We can then apply Lemma~\ref{lem_univalent} to obtain a value $\eps > 0$ such that any function $f$ holomorphic on $D(0, 1/9)$ and within $\eps$ of $\Psi$ on $\maindisc$ is univalent there. By decreasing $\eps$ slightly further, if necessary, we can guarantee that $ (D\cup B_0)\Subset f(D)$, and so the next part of the result follows. Disjointness of the preimages of $\mainreef$ follows from injectivity of $f$ together with the fact that $\maindisc$ and $\mainreef$ are disjoint. Note that by Lemma \ref{lem_univalent}, $\vert f'(z)\vert >3/2$ for all $z\in D$, and so its inverse map is a contraction with derivative bounded away from $1$. Thus, whenever $f(0)=0$, the final claim follows from the Banach fixed-point theorem.
\end{proof}

\subsection{Setting and idea of proof}
Let us fix $\maindisc$, $B_0$ and $\eps>0$ as in Lemma \ref{pullbacklemma}, so that the conclusions of the statement hold, and set 
\begin{equation}\label{eq_setA}
A\defeq D\left(-\frac14, \frac19 \right).
\end{equation}
Let $N_0\defeq 0$, and for each $j\in \N$, define $$N_j\defeq \sum_{i=1}^{j}n_i+m_i,$$
	where $(n_j)_{j \in \mathbb{N}}, (m_j)_{j \in \mathbb{N}}$ are given sequences of natural numbers, $n_0=m_0=0$, and $(m_j)_{j \in \mathbb{N}}$ is strictly increasing.
	In order to simplify our exposition, we will also assume that $m_1>1$.
	
	 Let  $U$ be a fixed regular domain whose closure $K\defeq \overline{U}$ is a full compact set. Using Lemma \ref{lem_compacts} we can choose a sequence of compact sets $(K_j)_{j\geq 0}$
	% and $(D_j)_{j\geq 0}$ 
	with the following properties:
	\begin{itemize}
		\item $K_{j+1}\subset \text{int}(K_{j})$ for all $j\in \N$,
		\item $K=\bigcap^\infty_{j=0} K_j$.
		%\item  $D_{j}\subset K_{j}$ such that $K_{j+1}\subset \text{int}(D_{j}).$ 
	\end{itemize}
	In addition, by applying an affine transformation, we may assume without loss of generality that $0\notin K_0$. We can similarly assume that
\begin{equation}\label{eq_K0}
\Phi^{j}(K_0)\Subset \maindisc, \text{ for } 0\leq j \leq n_1, \quad \text{ and }\quad  \Phi^{n_1+1}(K_0)\Subset B_0.
\end{equation}
%and
%\begin{equation}\label{eq_K0}
%\min\left\{ \min_{0\leq j< n_1}\left\{\dist(\Phi^{j}(K_0), \partial \maindisc)\right\}, \dist(\Phi^{n_1}(K_0), \partial B_0) \right\}>\frac{\eps}{m_1}. 
%\end{equation}

For each $j\in \N$, let us choose a full, compact, $2^{-j}$-dense subset $P_j\subset \partial K_j$. In other words, $P_j$ is such that for any $z\in \partial K_j$, $\dist(z,P_j)\leq 2^{-j}$. In particular, $\partial K$ is the Hausdorff limit of the sets $P_j$. These sets will belong to attracting basins of the map $f$ and are introduced to ensure that $\partial K\subset J(f)$, so that our wandering set has the prescribed shape.

In a rough sense, the orbit of $K$ under $f$ will be as follows. After iterating $n_1$ times in $D$, where  $f$ approximates $\Phi$, $f^{n_1}(K)$ will iterate $m_1$ times in translated copies of $B_0$, where $f$ acts approximately like the translation $z\mapsto z+1$. That is, $f^{n_1+\ell}(K)\subset B_0+\ell-1$ for $1\leq \ell\leq m_1$. In particular, $f^{N_1}(K)\subset B_0+m_1-1$. At this point, our construction will ensure that  $f^{N_1+\ell}(K)\subset D$ for $1\leq \ell \leq n_2$ while, again,  $f^{N_1+n_2+\ell}(K)\subset B_0+\ell-1$ for $1\leq \ell \leq m_2$. Note that the crucial steps  in the construction occur in the sets $B_0+m_j-1$, for $j\geq 1$. This is why assuming $m_1>1$ simplifies exposition without loosing significant generality. A finer analysis on where the iterates of $K$ lie when outside $D$ will be required, and thus, we will define inductively a collection of sets $(B_j)_{j\in \N}$, such that $B_j\subset B_0+m_j-1$ for $j\geq 1$; see Figure \ref{fig:main}. Additionally, we will use the following notation:
	$$\widehat{B}_j\defeq B_{j}+m_{j+1}-\max\{m_j,1\},\quad \quad \text{ for } j\geq 0,$$
and we will show that $\widehat{B}_j\Subset B_0+m_{j+1}-1$. 
In particular, we will be able to guarantee injectivity of our map $f$ on $D$ and on the collection of compact sets $(B_j)_{j\in \N}$ together with some of their translates; namely, on 
$$U_j\defeq\overline{D} \cup \bigcup^j_{l=0} \bigcup^{m_{l+1}-\max\{m_l,1\}-1}_{k=0}(\overline{B_l}+k), \quad \quad \text{ for } j\geq 0.$$

Finally, we set the collection of nested closed disks 
	$$\Delta_0\defeq \{0\} \quad  \text{ and } \quad \Delta_j\defeq \overline{D(0,m_j-1)}, \quad \text{ for all } j\geq 1.$$
	Note that, since $\eps<r_2/2<r_3/2 <1/6$, 
	\begin{equation}\label{eq_B0}
     B_0+m_j-2+\eps\Subset \Delta_j, \quad \text{ while } \quad B_0+m_j-1-\eps \Subset \C\setminus  \Delta_j \quad \text{ for each } j\geq 1.
	\end{equation}

\begin{figure}[htp]
	\centering
	\def\svgwidth{\linewidth}
	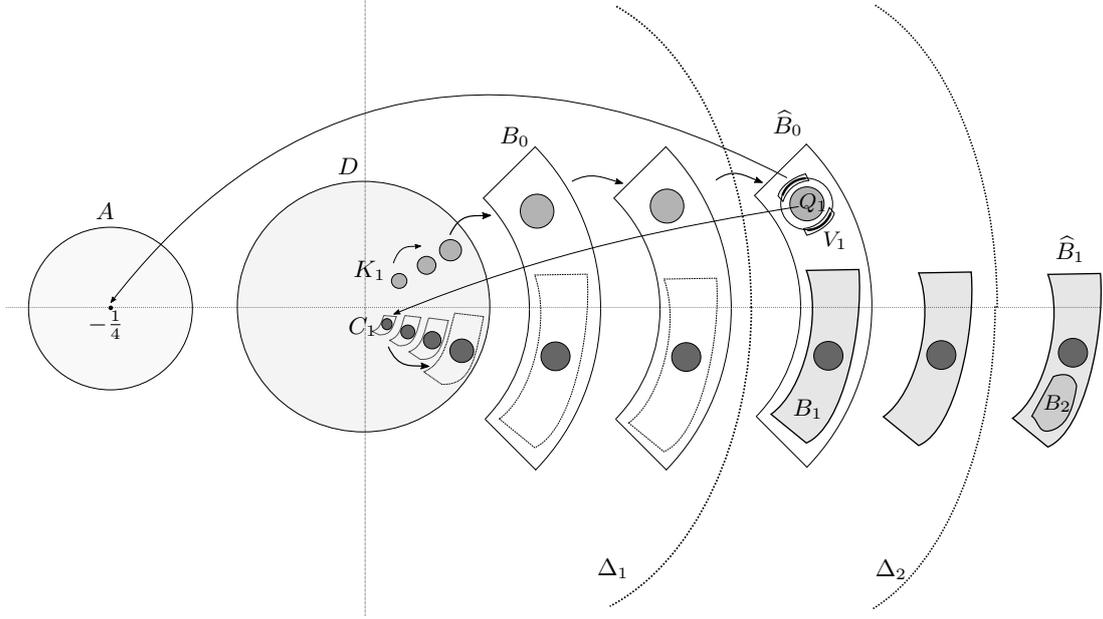
	\caption{Schematic of the sets and functions in the construction. }
	\label{fig:main}
\end{figure}

In order to prove Theorem \ref{maintheorem}, it suffices to show the following:
\begin{proposition}\label{prop_properties_f} There exists a transcendental entire function $f$ such that for all $j\geq 1$,
	\begin{enumerate}[(a)]
		\item \label{item_p1}$f(\overline{A})\cup f^{N_j+1}(P_{j-1}) \subset A$;
		%\item \label{item_p2} $\subset A$; 
		\item \label{item_p3} $f^{N_{j+1}}(K_j)\subset \C\setminus \Delta_j;$
		\item \label{item_p4} $\Re f^j(z)>-1$ for all $z\in \interior{(K_j)}$;
		\item \label{item_p5} $f^\ell(K_j)\subset D$ if and only if $N_p<\ell \leq N_{p}+n_{p+1}$ for some $0\leq p\leq j;$ 
		\item  \label{item_p6} $f(0)=0$;
		\item \label{item_p7} For each domain $\widehat{D}\subset D$ containing the origin, there exists $C\defeq C(\widehat{D})\in \N$ such that $f^\ell(K_j)\subset \widehat{D}$ if and only if $N_p< \ell \leq N_{p}+\max\{n_{p+1}-C, 0\}$ for some $0\leq p\leq j.$ 
	\end{enumerate}
\end{proposition}

Before proving Proposition \ref{prop_properties_f}, we show how Theorem~\ref{maintheorem} follows from this result.
\begin{proof}[Proof of Theorem \ref{maintheorem}, using Proposition \ref{prop_properties_f}]
By \eqref{item_p1} together with Montel's theorem, all points in $A \cup \bigcup_{j\geq 0} P_j$ are contained in attracting basins of $f$, and so they have bounded orbits.  On the other hand, by \eqref{item_p3} all points in $K$ have unbounded orbits. Consequently, any neighbourhood of any $z\in \partial K$ contains both points with bounded and unbounded orbits, which prevents the family of iterates to be normal on it. We conclude that $\partial K \subset J(f)$. By \eqref{item_p4} and Montel's theorem, $\interior(K)=U \subset F(f)$, as $U$ is regular. Then, since $U$ is connected, it is a Fatou component. By \eqref{item_p3} and \eqref{item_p5}, $U$ cannot be periodic or preperiodic, and so it must be a wandering domain. Finally, equation \eqref{aneq} in the statement of the theorem follows from~\eqref{item_p5}, concluding our proof.
\end{proof}	

\begin{remark} For the proof of Theorem \ref{maintheorem}, we do not require our function $f$ to satisfy items \eqref{item_p6} and \eqref{item_p7} in Proposition \ref{prop_properties_f}. In particular, a version of Proposition \ref{prop_properties_f} without \eqref{item_p6} and \eqref{item_p7} can be proved using Runge's classical approximation theorem instead of Theorem \ref{thm:Runge}. However, fixing points becomes necessary in the proof of Theorem~\ref{fancytheorem}. It is also required in the proof of Corollary \ref{cor_stronger}, which is a stronger version of Corollary \ref{c1}. Thus, seeking consistency, we have opted for using the same approach in all proofs.
\end{remark}

\subsection{Proof of Proposition \ref{prop_properties_f}}

Roughly speaking, we will obtain $f$ as the limit of a sequence of functions $(f_j)_{j\geq 0}$. In order to define them, we use a sequence of auxiliary functions $(\phi_j)_{j\geq 0}$, so that $f_{j}$ approximates $\phi_j$ in a neighbourhood of a compact set $T_j \subset \Delta_{j+1}$, up to an error $\eps_j$. In turn, $\phi_j$ has been defined to be $f_{j-1}$ in $\Delta_{j}$; see Figure \ref{fig:main}. In addition, we will later define inductively a collection of compact sets $(B_j)_{j\geq 1}$, so that our construction has the following properties, for each $j \geq 0$:
 \begin{enumerate}[(i)]
	\item \label{item_i} $\Delta_{j}\subset T_j\subset \Delta_{j+1}$;
		 \item \label{item_iv}${f^{N_{{j+1}}}_{j}(K_{j})}\Subset \widehat{B}_j \subset B_0+m_{j+1}-1$;
	\item \label{item_iii} $f_j^{N_{j+1}}$ is univalent on $K_j$;
	\item\label{item_ii} $f_j$ is univalent on $U_j\subset T_j$;
	\item \label{item_extra} $f_j^{N_j+n_{j+1}+\max\{m_l,1\}+k}(K_j)\Subset B_l+k\subset U_j\cup \widehat{B}_j$ for all $0\leq l\leq j$ and $0\leq k \leq  m_{l+1}-\max\{m_l,1\}$;
	\item \label{item_v}$\eps_j<\eps/4^{j}$;
	\item \label{item_vi} $f_j(0)=0$ and $f'_j(0)=3$.
\end{enumerate}

To start the induction, let $D$, $A$ and $B_0$ be the sets fixed in the previous subsection and set %$Q_0\defeq K_0$ %$\widehat{B}_0 = D(0, 1/2)$,
% $h_0\equiv \text{id}$,  
%and 
 $$T_0\defeq \overline{D} \cup \overline{A} \cup \bigcup^{m_1-2}_{k=0}(\overline{B_0}+k).$$ 
 Note that $T_0$ is a disjoint union of compact sets such that $\C\setminus T_0$ is connected. Consider the auxiliary function $\phi_0\colon T_0\to \C$ given by
\begin{equation*}
	\phi_0(z)\defeq
	\begin{cases}
		\Phi(z),& z \in \overline{D},\\
		-\frac14,& z \in \overline{A},\\
		z+1, &\text{otherwise}.
	\end{cases}
\end{equation*}

Observe that $\phi_0$ is holomorphic on $T_0$, $\phi^{N_1}_0$ is univalent on $K_0\subset D$, by \eqref{eq_K0}, and $\phi_0$ is univalent on $U_0$.

\begin{Claim} There exists $\eps_0<\eps/4$ such that if $g$ is any function that approximates $\phi_0$ up to an error $2\eps_0$, then
\begin{itemize} 
\item $g^{N_1}$ is univalent on $K_0$ and $g$ is univalent on $U_0$;
\item $g^\ell(K_0)\subset D$ for all $0\leq \ell\leq n_1$, while $g^{n_1+1+k}(K_0)\subset B_0+k$ for all $0\leq k \leq m_1-1$. 
\end{itemize}		
\end{Claim}
\begin{subproof}
Note that $\phi_0$ extends holomorphically to an open neighbourhood of $T_0$, so that $\phi_0$ is univalent on it. By applying Lemmas~\ref{lem_univalent} and \ref{lem_iterates} to this neighbourhood and $T_0$, the first part of the claim follows. By our assumption on $K_0$ in \eqref{eq_K0} and the definition of $\phi_0$, for all $\eps_0$ small enough the second part also holds.
\end{subproof}

We apply Theorem~\ref{thm:Runge} to obtain an entire map $f_0$ satisfying \eqref{item_vi} that approximates $\phi_0$ on $T_0$ up to an error $\eps_0$. In particular, by Claim 1, \eqref{item_iv}-\eqref{item_v} hold. Since $\Delta_{0}=\{0\}\subset D$ and by \eqref{eq_B0}, so does \eqref{item_i}. This concludes the first step on the induction argument.

Let $j\geq 1$, and suppose that $f_{j-1}$, $B_{j-1}$ and $\eps_{j-1}$ have been defined. Let us apply Lemma \ref{lem_compacts} to the set $K_j$ to find a compact set $L_j\subset K_{j-1}$ such that $K_j\subset \text{int}(L_j)$. We denote 
$$Q_j \defeq f^{N_j}_{j-1}(L_j). $$
Observe that $L_j$, and hence $Q_j$ are both compact and have an interior.
Since by the inductive hypotheses $f^{N_j}_{j-1}$ is injective on $K_{j-1}$ and $f^{N_j}_{j-1}(K_{j-1})\subset \widehat{B}_{j-1}$, we have that $f_{j-1}^{N_j}(P_{j-1})\Subset \widehat{B}_{j-1}$ and  $f_{j-1}^{N_j}(P_{j-1})\cap Q_j=\emptyset$. We can then choose a neighbourhood of $f_{j-1}^{N_j}(P_{j-1})$ compactly contained in $\widehat{B}_{j-1}$ and disjoint from $Q_j$. We denote the closure of this neighbourhood by $V_{j}$.
%noting that $Q_j \subset \widehat{B}_{j-1} $.

The following claim provides us with a set $C_j\subset D$ where $Q_j$, and so $f_{j-1}^{N_j}(K_j)$, will be mapped to under our next model map.
\begin{Claim} There exists a compact set $C_j\subset D$ with non-empty interior % and $\delta_j>0$ 
	such that % if $g$ is any function approximating $f_{j-1}$ up to an error $\delta_j$, then
the following hold:
\begin{itemize}
%\item$ g^\ell(C_j)\subset D\setminus \bigcup^{n_1}_{\ell=0}g^\ell(h_{j-1}(C_{j-1}))$ for all $0\leq \ell\leq n_1$;
%\item $g^\ell(C_j)\subset (B_0+\ell)\setminus f_j^\ell(K_0)$ for all $n_1<\ell<m_1$.
%\item \todo{no?} $g^{N_{j-1}+1+\ell}(K_{j-1})$ and $g^{n_{j+1}-n_j+\ell}(C_j)$ are disjoint for all $0\leq \ell \leq n_j+m_j$.
%\item $f_{j-1}^{N_{j-1}+m_j-1}(K_{j-1})$ and $f_{j-1}^{n_{j+1}+m_j-1}(C_j)$ are disjoint.
\item $f_{j-1}^{n_{j+1}-1+m_j}(C_j)\Subset \widehat{B}_{j-1}\setminus (Q_j\cup V_{j})$. 
\item $f_{j-1}^{n_{j+1}-1+m_j}$ is univalent on $C_j$;
\item $f_{j-1}^{\ell}(C_j)\subset D$ for $0\leq \ell < n_{j+1},$ while  $f_{j-1}^{n_{j+1}-1+\max\{m_l,1\} +k}(C_j)\Subset B_l+k$ for all $0\leq l\leq j$ and $0\leq k\leq m_{l+1}-\max\{m_l,1\}$.
\end{itemize}
\end{Claim}
\begin{subproof}
The idea of the proof is as follows. We want to pull back a subset of $\widehat{B}_{j-1}\setminus (Q_j\cup V_{j})$, using the right inverse branches of $f_{j-1}$, $m_j+(n_{j+1}-1)$ times. We shall do so in such a way that the iterated preimages remain in $U_{j-1}$, where we can guarantee injectivity of  $f_{j-1}$.

To start with, by injectivity of $f_{j-1}$ in $U_{j-1}$, for each $0\leq l<j$ and $0\leq k< m_{l+1}-\max\{m_l,1\}$, $f_{j-1}\vert_{B_{l}+k}$ is a conformal isomorphism to its image. Thus, we can consider the restrictions of the inverse branches 
$$F_{l,k}\defeq  \left(f_{j-1}\vert_{B_{l}+k}\right)^{-1}\colon f_{j-1}(B_{l}+k)\cap (B_{l}+k+1) \to B_{l}+k,$$
whose domains and codomains lie in $U_{j-1}\cup \widehat{B}_{j-1}$ and are non-empty by \eqref{item_extra}. Set 
$$E\defeq f_{j-1}(B_{j-1}+m_j-m_{j-1}-1) \cap \widehat{B}_{j-1}\setminus (Q_j\cup V_{j}),$$
which is non-empty and has non-empty interior by \eqref{item_extra} and the definition of $V_j$ and $Q_j$. Define the map
$$F\defeq(F_{0,0}\circ \cdots \circ F_{j-1,m_{j}-{m_{j-1}}-1})\colon E \to B_0,$$ which is a conformal isomorphism to its image. By  \eqref{item_extra}, $F(f^{N_j}_{j-1}(K_{j-1}))$ is well defined and equals $f^{N_{j-1}+n_{j}+1}_{j-1}(K_{j-1})$. Hence, $F(E)$ is non-empty, and so there exists a neighbourhood of $f^{N_j}_{j-1}(K_{j-1})$ in $E$ whose image under $F$ is in $B_0.$ Let us choose some closed set $X$ with non-empty interior in such neighbourhood minus $f^{N_j}_{j-1}(K_{j-1})$. In particular,
\begin{equation}\label{eq_FX}
 F(X)\subset B_0\setminus f^{N_{j-1}+n_{j}+1}_{j-1}(K_{j-1}).
\end{equation} 

Next, using Lemma \ref{pullbacklemma}, there exists an $n_{j+1}$-inverse branch of $f_{j-1}$, 
$$G\defeq \left(f_{j-1}\vert\right)^{-n_{j+1}}\colon B_0\to D,$$ which is a conformal isomorphism to its image. Then, by injectivity of $f_{j-1}$ in $D$, Lemma \ref{pullbacklemma} and \eqref{eq_FX}, the claim follows choosing $C_j\defeq G(F(X))$.
\end{subproof}	

Denote
\begin{equation*}
	B_j\defeq f_{j-1}^{n_{j+1}-1+m_j}(C_j),
\end{equation*}
%\leti{
%$$C_j\defeq \left(f_{j-1}^{n_j+m_j}\right)^{-1}(B_j).$$
%We need $\diam (C_j)\geq (m_j+n_j)\eps_{j-1}$, since otherwise I cannot find a map $h_j$ such that $h_j(Q_j)\subset C_j$  and so that}           
and let $h_j\colon Q_j\to C_j$ be a non-constant affine contraction such that 
\begin{equation*}
h_j(Q_j) \Subset C_j.
\end{equation*}
%\begin{equation*}
%\dist\left(h_j(Q_j), \partial \left(f_{j-1}^{n_j+m_j}\right)^{-1}(B_j)\right)>\eps_{j}/(n_j+m_j).
%\end{equation*} 
Finally, set 
$$T_{j}\defeq \Delta_{j}\cup V_{j}\cup Q_{j} \cup \bigcup^{m_{j+1}-m_{j}-1}_{k=0} (B_{j}+k),$$ 
which is a disjoint union of compact sets such that $\C\setminus T_{j}$ is connected. Since $B_j\subset \widehat{B}_{j-1}$ and using \eqref{eq_B0}, 
\begin{equation}\label{eq_inclusion_Bj}
B_j+m_{j+1}-m_j-1\subset B_0+m_{j+1}-2 \subset \Delta_j,
\end{equation}
and \eqref{item_i} follows. We define $\phi_{j}\colon T_{j}\to \C$ as
\begin{equation*}
	\phi_{j}(z)\defeq
		\begin{cases}
			f_{j-1}(z),& z \in \Delta_{j},\\
			-\frac14,&  z \in V_{j},\\
			h_j(z),& z \in Q_{j},\\
			z+1,& \text{otherwise},
		\end{cases}
	\end{equation*}
    noting that $\phi_{j}$ is holomorphic on $T_{j}$. % and injective on $U_j$ by the inductive hypothesis \eqref{item_ii} and by definition.  
    Note that 
    \begin{equation}\label{eq_phi_in_Kj}
    \phi_j^{N_{j+1}}\vert_{K_j}=(\tau^{m_{j+1}-m_j}\circ f^{n_{j+1}-1+m_j}_{j-1}\circ h_j\circ f^{N_j}_{j-1})\vert_{K_j},
    \end{equation}
    where $\tau$ is the translation map $z\mapsto z+1$. Hence, by \eqref{item_iii}, the definition of $h_j$ and Claim 2, $\phi_j^{N_{j+1}}$ is univalent on $K_j$ as a composition of univalent maps. In addition, by \eqref{item_ii} and definition, $\phi_j$ is univalent on $U_{j-1}\cup \bigcup^{m_j+1-m_j-1}_{k=0} B_j=U_j$, and
    \begin{equation}\label{eq_phijPj}
    \phi_j^{N_j+1}(P_{j-1})=\phi_j\circ f_{j-1}^{N_j}(P_{j-1})\subset \phi_j(V_j)=\{-1/4\}\subset A.
    \end{equation}
    
    \begin{Claim}\label{c3}There exists $\eps_j<\eps_{j-1}/4$ such that any function $g$ that approximates $\phi_j$ up to an error $2\eps_j$ satisfies the following: 
    	\begin{enumerate}
    		\item \label{P_j_in_A} $g^{N_j+1}(P_{j-1}) \subset A$; 
    		\item \label{item_univalent}$g^{N_{j+1}}$ is univalent on $K_j$ and $g$ is univalent on $U_j$. 
    		\item \label{in_D} $g^{N_{j}+\ell}(K_j)\subset D$ for $1 \leq \ell \leq n_{j+1},$ while  $g^{N_j+n_{j+1}+\max\{m_l,1\}+k}(K_j)\Subset B_l+k$ for all $0\leq l\leq j$ and $0\leq k \leq m_{l+1}-\max\{m_l,1\}$.
    		\item \label{item_4}$g^{N_{j+1}}(K_j)\Subset \widehat{B}_j$.
    	\end{enumerate}
    \end{Claim}
   
     \begin{subproof}
     	By \eqref{eq_phi_in_Kj} and \eqref{eq_phijPj}, items \eqref {P_j_in_A} and \eqref{item_univalent} hold for $g=\phi_j$, and so they are possible by Lemmas \ref{lem_univalent} and \ref{lem_iterates}. Note that by definition, $$\phi^{N_j+1}_j(K_j)\subset \phi_j\circ f^{N_j}_{j-1}(L_j)=\phi_j^{N_j+1}(L_j) \Subset C_j.$$
     	By this, Claim 2, and the definition of $\phi_j$, \eqref{in_D} and \eqref{item_4} hold for $g=\phi_j$. Hence, reducing $\eps_j$ if necessary, items \eqref{in_D} and \eqref{item_4} follow. 
    \end{subproof}
    
We apply Theorem~\ref{thm:Runge} to obtain an entire function $f_j$ satisfying \eqref{item_vi} that approximates $\phi_j$ in $T_j$, up to an error of at most $\eps_j$. By Claim 3, the first part of \eqref{item_iv} and \eqref{item_iii}-\eqref{item_v} hold. We are left to check that $\widehat{B}_j \subset B_0+m_{j+1}-1$, which follows from \eqref{eq_B0} and \eqref{eq_inclusion_Bj}. This concludes the inductive construction.

By our choice of the sequence $(\eps_j)$ satisfying \eqref{item_v}, we have that $(f_k)^\infty_{k=j}$ is a Cauchy sequence when restricted to the set $T_j$. Since, by \eqref{item_i}, $\Delta_j\subset T_j$ and $\bigcup^{\infty}_{j=1}\Delta_j=\C$, given the assumption of $(m_j)$ being strictly increasing, we have that the functions $f_j$ converge locally uniformly to an entire function $f$.

We are left to check that $f$ satisfies the properties in the statement of the proposition. Note that for any $j\geq 0$ and $z\in T_j$, 
\begin{equation*}
	\vert f(z)-\phi_j(z)\vert \leq \sum^\infty_{k=j} \vert f_{k+1}(z)-f_{k}(z)\vert+ \vert f_j(z)-\phi_j(z)\vert \leq \sum^\infty_{k=j} \eps_k\leq 2\eps_j,
\end{equation*}
and so the conclusions in Claims 1 and \ref{c3} hold for $f$. Since $\phi_0(A)=f_0(A)=-\frac{1}{4}$, we have that $f(A)\subset D(-\frac{1}{4},\frac{1}{18})\subset A$. By this and item \eqref{P_j_in_A} in Claim 3, \eqref{item_p1} follows. Moreover, \eqref{item_p3} holds by \eqref{item_4} and \eqref{eq_B0}, and \eqref{item_p4} and \eqref{item_p5} follow from \eqref{in_D}. Finally, \eqref{item_p6} is a consequence of \eqref{item_vi}, and \eqref{item_p7} follows from \eqref{in_D} together with Lemma \ref{pullbacklemma}. This concludes the proof of Proposition \ref{prop_properties_f}.

%\newpage 
%.
%\newpage
\subsection{Proof of Corollary \ref{c1}}
The fact that $f$ is constructed to fix zero allows us to prove the following stronger version of Corollary \ref{c1}.
\begin{cor}\label{cor_stronger} Suppose that $\lambda \in [0, 1]$, let $U$ be a regular domain whose closure in $\C$ is a full compact set, and choose a point $p\in \C\setminus \overline{U}$. Then there is a transcendental entire function $f$ for which $U$ is a wandering domain and such that for any sufficiently small neighbourhood $\widehat{D}$ of $p$, we have that
	\begin{equation}\label{anothereq2}
		\lim_{k\rightarrow\infty} \frac{\#\{ n \leq k : f^n(U) \subset \widehat{D} \}}{k} = \lambda.
	\end{equation}
\end{cor}
\begin{proof}
By applying an affine transformation, we may assume without loss of generality that $p=0$. Let us choose the sequences $m_j=j$, for $j\in \N$, and 
\begin{equation*}
	n_j\defeq
	\begin{cases}
		j^2,& \lambda=1,\\
		\left\lceil\frac{\lambda}{1-\lambda}\cdot j \right\rceil, &  \text{otherwise,}\\
	\end{cases}
\end{equation*}
and let us apply Theorem \ref{maintheorem} for the given domain $U$ and this choice of sequences. Let $\widehat{D}\subset D$ be any other domain containing $0$. Then, by Proposition \ref{prop_properties_f}\eqref{item_p7}, there exists $C \in \N$, depending on $\widehat{D}$, such that equation \eqref{aneq} in Theorem \ref{maintheorem} holds for the sequences 
\begin{equation*}
\widehat{n}_j\defeq \max \{0,n_j -C\},  \quad \quad \widehat{m}_j\defeq m_j +\min\{n_j,C\}, \quad \text{ for each } j\in \N,
\end{equation*}
with $\widehat{D}$ taking the role of $D$.

For each $k\in \N$, denote $\Delta(k) \defeq \#\{ n \leq k : f^n(U) \subset \widehat{D} \}/k$. We want to show that $\lim_{k\to \infty}\Delta(k)=\lambda$. Observe that for each $k\in \N$, there exists $p_k \in \N$ such that $\widehat{N}_{p_k}\leq k<\widehat{N}_{p_k+1}$, where for each $j\in \N$,  $\widehat{N}_j\defeq \sum_{i=1}^{j}\widehat{n}_i+\widehat{m}_i.$ Hence,
\begin{equation*}
	\frac{\sum_{j=1}^{p_k}\widehat{n}_j}{\widehat{N}_{p_k+1}}\leq \Delta(k)\leq\frac{\sum_{j=1}^{p_k+1}\widehat{n}_j}{\widehat{N}_{p_k}+\widehat{n}_{p_k+1}}.
\end{equation*}

%\begin{equation*}
%\frac{\sum_{j=1}^{p_k}\widehat{n}_j}{\sum_{j=1}^{p_k}\widehat{n}_j+\sum_{j=1}^{p_k}\widehat{m}_j}\leq \Delta(k)<\frac{\sum_{j=1}^{p_k}\widehat{n}_j}{\sum_{j=1}^{p_k}\widehat{n}_j+\widehat{n}_{p_k+1}+\sum_{j=1}^{p_k}\widehat{m}_j+\widehat{m}_{p_k+1}}.
%\end{equation*}

Equivalently, 
\begin{equation}\label{eq_Lambdak}
	\frac{1}{1+\frac{\widehat{n}_{p_k+1}+\sum_{j=1}^{p_k+1}\widehat{m}_j}{\sum_{j=1}^{p_k}\widehat{n}_j}}\leq \Delta(k) \leq \frac{1}{1+\frac{\sum_{j=1}^{p_k}\widehat{m}_j}{\sum_{j=1}^{p_k+1}\widehat{n}_j}}.
\end{equation}
Now, 
\begin{align*}
 \lim_{p\to \infty}\frac{\sum_{j=1}^{p}\widehat{m}_j}{\sum_{j=1}^{p+1}\widehat{n}_j}&= \lim_{p\to \infty}\frac{\sum_{j=1}^{p}\widehat{m}_j+\widehat{n}_{p+1}+\widehat{m}_{p+1}}{\sum_{j=1}^{p}\widehat{n}_j}\\ 
	&=\begin{cases}
	 \lim_{p\to \infty}\frac{p(p+1)/2+pC+(p+1)^2+(p+1)}{p(p+1)(2p+1)/6-pC}=0,& \lambda=1,\\	
		 \lim_{p\to \infty}\frac{p(p+1)/2+pC+\frac{\lambda}{1-\lambda}\cdot(p+1)+(p+1)}{\frac{\lambda}{1-\lambda}\cdot p(p+1)/2+p-pC}=\frac{1-\lambda}{\lambda}, &  \text{otherwise},\\
\end{cases}
\end{align*}
and so, since $p_k\to \infty$ as $k\to \infty$, \eqref{anothereq2} follows from this together with \eqref{eq_Lambdak}.
\end{proof}

\section{Sketch of proof of Theorem~\ref{fancytheorem}}

For this result we need to modify the construction in Theorem~\ref{maintheorem}. Let us fix a domain $U$, a collection of points $(z_l)_{1 \leq l \leq p}$ and numbers $(\lambda_l)_{1 \leq l \leq p}$ as in the statement of the theorem, with $p\geq 2$. We may assume without loss of generality that $\sum^p_{l=1} \lambda_l=1$, since otherwise we can create an extra domain $D_{p+1}$ and $\lambda_{p+1}\defeq 1-\sum^p_{l=1}\lambda_l$. For each $1 \leq l \leq p$, let $D_l$ be a translated copy of the disk $\maindisc$ provided by Lemma~\ref{pullbacklemma}, centred at $z_l$, and let $B^l_0$ be the corresponding translation of the set $B_0$. For simplicity, we will assume that the sets
$$\{D_l, B^l_0+k, \text{ for some } l\leq p, k\geq 0\}$$ have pairwise disjoint closures, and that $U\subset D_1$ and $\Phi(U)\subset B^1_0$. Otherwise, we simply scale and rotate them, and modify the translations in our construction appropriately.

For each $1 \leq l \leq p$ and $j\in \N$, define
	\begin{equation*}
		\quad m^l_j\defeq j, \quad \text{ and } \quad 	n^l_j\defeq \left\lceil \lambda_l\cdot j^2 \right\rceil.
	\end{equation*}
	Our domain $U$ will spend $n^1_1$ iterates inside $D_1$, $m^1_1$ outside $\bigcup^p_{j=1}D_j$, $n^2_1$ iterates inside $D_2$, and so on. In a rough sense, the proof proceeds essentially as in section~\ref{sec_3}, with $p$ copies of each of the sets defined in the proof of Proposition~\ref{prop_properties_f}. The main novelty is that $h^l_j(Q^l_j)\subset C^{l+1}_{j-1}$ for $l<p$, and $h^p_j(Q^p_j)\subset C^{1}_j$; see Figure~\ref{fig_13}. Hence, each step in the previous construction is replaced by a cyclic one, where $U$ passes through all $D_l$ before returning to $D_1$. In addition, our maps $f_j$ will fix all the points $z_j$, and the discs $\Delta_j$ are replaced by squares of side-length $m^1_j-1$.

	\begin{figure}[htp]
		\centering
		\def\svgwidth{0.99\linewidth}
		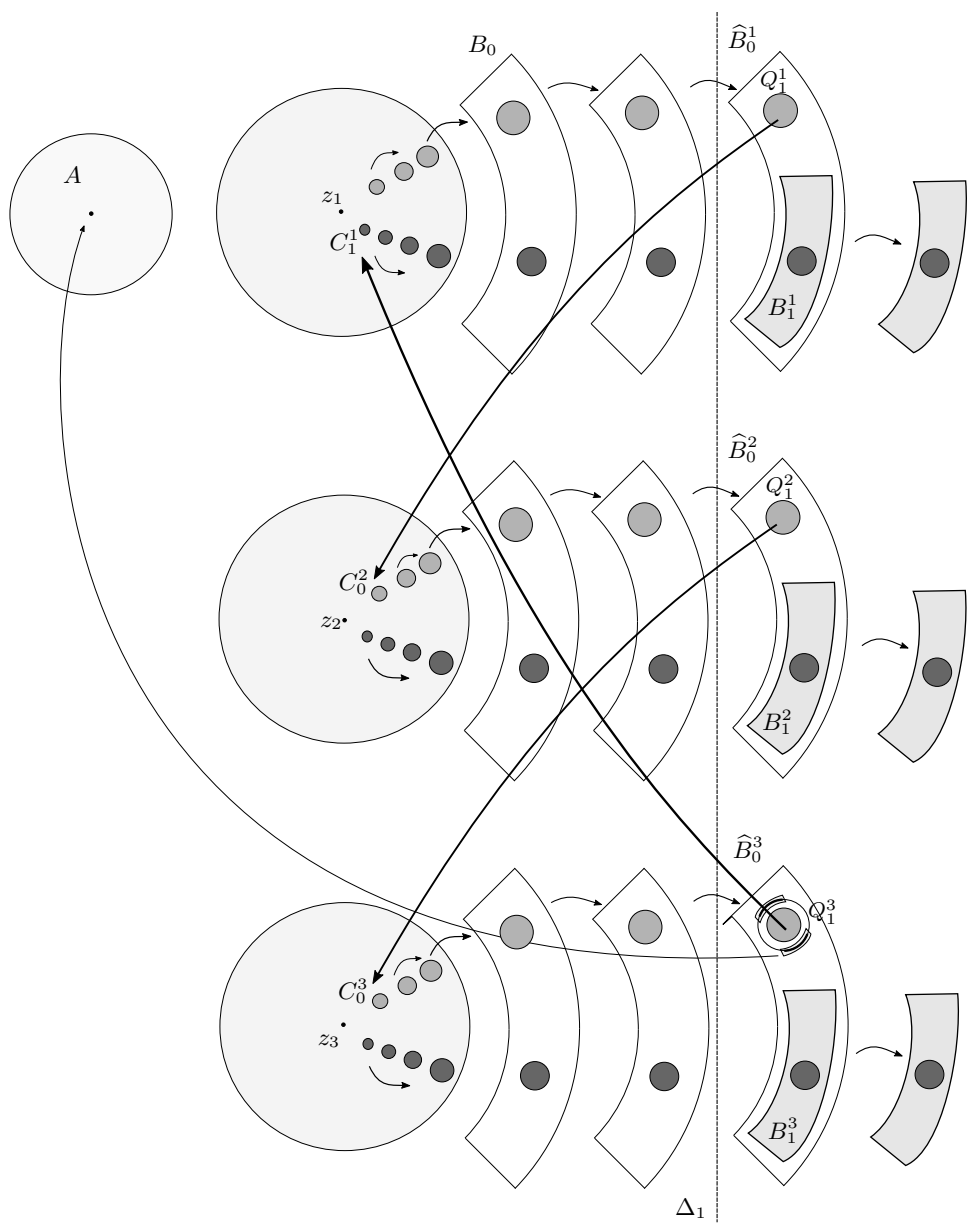
		\caption{Schematic of the sets and functions in the construction of $f$ satisfying Theorem \ref{fancytheorem}.}
		\label{fig_13}
	\end{figure}

	We claim that the limit function $f$ built this way satisfies the requirements of the theorem. To see that, for each $k\in \N$ and $1\leq l\leq p$, denote $\Delta^l(k) \defeq \#\{ n \leq k : f^n(U) \subset D_l\}/k$. We want to show that $\lim_{k\to \infty}\Delta^l(k)=\lambda_l$. For each $j\in \N$,  let $N_j\defeq \sum_{l=0}^{p}\sum_{i=0}^{j}n^l_i+\widehat{m}^l_i,$ and observe that for each $k\in \N$, there exists $q_k \in \N$ such that $N_{q_k}\leq k<N_{q_k+1}$. Hence, 
	\begin{equation*}
		\frac{\sum_{j=1}^{q_k}\lceil\lambda_l\cdot j^2\rceil}{p\sum^{q_k+1}_{j=0} j+\sum^{q_k+1}_{j=0} j^2}\approx \frac{\sum_{j=1}^{q_k}n^l_j}{N_{q_k+1}}\leq \Delta^l(k)\leq \frac{\sum_{j=1}^{q_k+1}n^l_j}{N_{q_k}+n^l_{q_k+1}}\approx \frac{\sum_{j=1}^{q_k+1}\lceil\lambda_l\cdot j^2\rceil}{p\sum^{q_k}_{j=0} j+\sum^{q_{k}+1}_{j=0} j^2 },
	\end{equation*}
	where we have used that $\sum^p_{l=1} \lambda_i=1$.
	Then, a calculation similar to the one in the proof of Corollary \ref{cor_stronger} yields the desired bounds.
	
	\bibliographystyle{alpha}
	\bibliography{biblioNearlyboundedWD}
\end{document}

%% file: main_construction.pdf_tex
%% Creator: Inkscape 1.2.2 (b0a8486, 2022-12-01), www.inkscape.org
%% Accompanies image file 'main_construction.pdf' (pdf, eps, ps)
\begingroup%
  \makeatletter%
  \providecommand\color[2][]{%
    \errmessage{(Inkscape) Color is used for the text in Inkscape, but the package 'color.sty' is not loaded}%
    \renewcommand\color[2][]{}%
  }%
  \providecommand\transparent[1]{%
    \errmessage{(Inkscape) Transparency is used (non-zero) for the text in Inkscape, but the package 'transparent.sty' is not loaded}%
    \renewcommand\transparent[1]{}%
  }%
  \providecommand\rotatebox[2]{#2}%
  \newcommand*\fsize{\dimexpr\f@size pt\relax}%
  \newcommand*\lineheight[1]{\fontsize{\fsize}{#1\fsize}\selectfont}%
  \ifx\svgwidth\undefined%
    \setlength{\unitlength}{629.29133858bp}%
    \ifx\svgscale\undefined%
      \relax%
    \else%
      \setlength{\unitlength}{\unitlength * \real{\svgscale}}%
    \fi%
  \else%
    \setlength{\unitlength}{\svgwidth}%
  \fi%
  \global\let\svgwidth\undefined%
  \global\let\svgscale\undefined%
  \makeatother%
  \begin{picture}(1,0.55855856)%
    \lineheight{1}%
    \setlength\tabcolsep{0pt}%
    \put(0,0){\includegraphics[width=\unitlength,page=1]{main_construction.pdf}}%
    \put(0.08092298,0.35738836){\color[rgb]{0,0,0}\makebox(0,0)[lt]{\lineheight{1.25}\smash{\begin{tabular}[t]{l}\fontsize{9pt}{1em}$A$\end{tabular}}}}%
    \put(0.29734377,0.39825374){\color[rgb]{0,0,0}\makebox(0,0)[lt]{\lineheight{1.25}\smash{\begin{tabular}[t]{l}\fontsize{9pt}{1em}$D$\end{tabular}}}}%
    \put(0.44259244,0.42552967){\color[rgb]{0,0,0}\makebox(0,0)[lt]{\lineheight{1.25}\smash{\begin{tabular}[t]{l}\fontsize{9pt}{1em}$B_0$\end{tabular}}}}%
    \put(0.70504294,0.17996486){\color[rgb]{0,0,0}\makebox(0,0)[lt]{\lineheight{1.25}\smash{\begin{tabular}[t]{l}\fontsize{9pt}{1em}$B_1$\end{tabular}}}}%
    \put(0.92900654,0.18601797){\color[rgb]{0,0,0}\makebox(0,0)[lt]{\lineheight{1.25}\smash{\begin{tabular}[t]{l}\fontsize{8pt}{1em}$B_2$\end{tabular}}}}%
    \put(0.71033936,0.36685342){\color[rgb]{0,0,0}\makebox(0,0)[lt]{\lineheight{1.25}\smash{\begin{tabular}[t]{l}\fontsize{8pt}{1em}$Q_1$\end{tabular}}}}%
    \put(0.73166127,0.33269117){\color[rgb]{0,0,0}\makebox(0,0)[lt]{\lineheight{1.25}\smash{\begin{tabular}[t]{l}\fontsize{8pt}{1em}$V_1$\end{tabular}}}}%
    \put(0.68688379,0.43495047){\color[rgb]{0,0,0}\makebox(0,0)[lt]{\lineheight{1.25}\smash{\begin{tabular}[t]{l}\fontsize{9pt}{1em}$\widehat{B}_0$\end{tabular}}}}%
    \put(0.93963497,0.32156913){\color[rgb]{0,0,0}\makebox(0,0)[lt]{\lineheight{1.25}\smash{\begin{tabular}[t]{l}\fontsize{9pt}{1em}$\widehat{B}_1$\end{tabular}}}}%
    \put(0.31162066,0.30522489){\color[rgb]{0,0,0}\makebox(0,0)[lt]{\lineheight{1.25}\smash{\begin{tabular}[t]{l}\fontsize{9pt}{1em}$K_1$\end{tabular}}}}%
    \put(0.30700065,0.25411499){\color[rgb]{0,0,0}\makebox(0,0)[lt]{\lineheight{1.25}\smash{\begin{tabular}[t]{l}\fontsize{9pt}{1em}$C_1$\end{tabular}}}}%
    \put(0.0740793,0.25689257){\color[rgb]{0,0,0}\makebox(0,0)[lt]{\lineheight{1.25}\smash{\begin{tabular}[t]{l}\fontsize{9pt}{1em}$-\frac{1}{4}$\end{tabular}}}}%
    \put(0.52984562,0.03620448){\color[rgb]{0,0,0}\makebox(0,0)[lt]{\lineheight{1.25}\smash{\begin{tabular}[t]{l}\fontsize{9pt}{1em}$\Delta_1$\end{tabular}}}}%
    \put(0.77861433,0.0354478){\color[rgb]{0,0,0}\makebox(0,0)[lt]{\lineheight{1.25}\smash{\begin{tabular}[t]{l}\fontsize{9pt}{1em}$\Delta_2$\end{tabular}}}}%
  \end{picture}%
\endgroup%

%% file: Theorem13.pdf_tex
%% Creator: Inkscape 1.2.2 (b0a8486, 2022-12-01), www.inkscape.org
%% Accompanies image file 'Theorem13.pdf' (pdf, eps, ps)
\begingroup%
  \makeatletter%
  \providecommand\color[2][]{%
    \errmessage{(Inkscape) Color is used for the text in Inkscape, but the package 'color.sty' is not loaded}%
    \renewcommand\color[2][]{}%
  }%
  \providecommand\transparent[1]{%
    \errmessage{(Inkscape) Transparency is used (non-zero) for the text in Inkscape, but the package 'transparent.sty' is not loaded}%
    \renewcommand\transparent[1]{}%
  }%
  \providecommand\rotatebox[2]{#2}%
  \newcommand*\fsize{\dimexpr\f@size pt\relax}%
  \newcommand*\lineheight[1]{\fontsize{\fsize}{#1\fsize}\selectfont}%
  \ifx\svgwidth\undefined%
    \setlength{\unitlength}{629.29133858bp}%
    \ifx\svgscale\undefined%
      \relax%
    \else%
      \setlength{\unitlength}{\unitlength * \real{\svgscale}}%
    \fi%
  \else%
    \setlength{\unitlength}{\svgwidth}%
  \fi%
  \global\let\svgwidth\undefined%
  \global\let\svgscale\undefined%
  \makeatother%
  \begin{picture}(1,1.12612613)%
    \lineheight{1}%
    \setlength\tabcolsep{0pt}%
    \put(0,0){\includegraphics[width=\unitlength,page=1]{Theorem13.pdf}}%
    \put(0.0690048,0.97024487){\color[rgb]{0,0,0}\makebox(0,0)[lt]{\lineheight{1.25}\smash{\begin{tabular}[t]{l}\fontsize{9pt}{1em}$A$\end{tabular}}}}%
    \put(0.43477936,1.08997944){\color[rgb]{0,0,0}\makebox(0,0)[lt]{\lineheight{1.25}\smash{\begin{tabular}[t]{l}\fontsize{9pt}{1em}$B_0$\end{tabular}}}}%
    \put(0.70573703,0.85060991){\color[rgb]{0,0,0}\makebox(0,0)[lt]{\lineheight{1.25}\smash{\begin{tabular}[t]{l}\fontsize{9pt}{1em}$B^1_1$\end{tabular}}}}%
    \put(0.70077866,0.47253466){\color[rgb]{0,0,0}\makebox(0,0)[lt]{\lineheight{1.25}\smash{\begin{tabular}[t]{l}\fontsize{9pt}{1em}$B^2_1$\end{tabular}}}}%
    \put(0.70613793,0.10004274){\color[rgb]{0,0,0}\makebox(0,0)[lt]{\lineheight{1.25}\smash{\begin{tabular}[t]{l}\fontsize{9pt}{1em}$B^3_1$\end{tabular}}}}%
    \put(0.29899051,0.18551445){\color[rgb]{0,0,0}\makebox(0,0)[lt]{\lineheight{1.25}\smash{\begin{tabular}[t]{l}\fontsize{9pt}{1em}$z_3$\end{tabular}}}}%
    \put(0.62222523,0.03014701){\color[rgb]{0,0,0}\makebox(0,0)[lt]{\lineheight{1.25}\smash{\begin{tabular}[t]{l}\fontsize{9pt}{1em}$\Delta_1$\end{tabular}}}}%
    \put(0.30119002,0.56634031){\color[rgb]{0,0,0}\makebox(0,0)[lt]{\lineheight{1.25}\smash{\begin{tabular}[t]{l}\fontsize{9pt}{1em}$z_2$\end{tabular}}}}%
    \put(0.69896054,1.05895839){\color[rgb]{0,0,0}\makebox(0,0)[lt]{\lineheight{1.25}\smash{\begin{tabular}[t]{l}\fontsize{8pt}{1em}$Q^1_1$\end{tabular}}}}%
    \put(0.7042353,0.68872503){\color[rgb]{0,0,0}\makebox(0,0)[lt]{\lineheight{1.25}\smash{\begin{tabular}[t]{l}\fontsize{8pt}{1em}$Q^2_1$\end{tabular}}}}%
    \put(0.74246604,0.30167224){\color[rgb]{0,0,0}\makebox(0,0)[lt]{\lineheight{1.25}\smash{\begin{tabular}[t]{l}\fontsize{8pt}{1em}$Q^3_1$\end{tabular}}}}%
    \put(0.66957913,1.09402047){\color[rgb]{0,0,0}\makebox(0,0)[lt]{\lineheight{1.25}\smash{\begin{tabular}[t]{l}\fontsize{9pt}{1em}$\widehat{B}^1_0$\end{tabular}}}}%
    \put(0.66953091,0.72054674){\color[rgb]{0,0,0}\makebox(0,0)[lt]{\lineheight{1.25}\smash{\begin{tabular}[t]{l}\fontsize{9pt}{1em}$\widehat{B}^2_0$\end{tabular}}}}%
    \put(0.31813606,0.59983098){\color[rgb]{0,0,0}\makebox(0,0)[lt]{\lineheight{1.25}\smash{\begin{tabular}[t]{l}\fontsize{9pt}{1em}$C^2_0$\end{tabular}}}}%
    \put(0.31752191,0.22730123){\color[rgb]{0,0,0}\makebox(0,0)[lt]{\lineheight{1.25}\smash{\begin{tabular}[t]{l}\fontsize{9pt}{1em}$C^3_0$\end{tabular}}}}%
    \put(0.30215375,0.95378096){\color[rgb]{0,0,0}\makebox(0,0)[lt]{\lineheight{1.25}\smash{\begin{tabular}[t]{l}\fontsize{9pt}{1em}$z_1$\end{tabular}}}}%
    \put(0.30961711,0.90932667){\color[rgb]{0,0,0}\makebox(0,0)[lt]{\lineheight{1.25}\smash{\begin{tabular}[t]{l}\fontsize{9pt}{1em}$C^1_1$\end{tabular}}}}%
    \put(0.67433466,0.35601832){\color[rgb]{0,0,0}\makebox(0,0)[lt]{\lineheight{1.25}\smash{\begin{tabular}[t]{l}\fontsize{9pt}{1em}$\widehat{B}^3_0$\end{tabular}}}}%
  \end{picture}%
\endgroup%